\documentclass[12pt,regno]{amsart}
\usepackage{amsfonts,amsmath,amssymb,amsbsy,amstext,amsthm}
\usepackage{color,enumerate,float,fullpage,verbatim}
\usepackage[shortlabels]{enumitem}
\usepackage{euler, amsfonts, amssymb, latexsym, epsfig,epic}

\usepackage{url}
\usepackage[colorlinks=true,hyperindex, linkcolor=Mahogany, pagebackref=false, citecolor=NavyBlue]{hyperref}
\usepackage[dvipsnames]{xcolor}
\usepackage{verbatim}
\usepackage{color}
\usepackage[all]{xy}  
\usepackage{mathrsfs}

\usepackage{stmaryrd}

\theoremstyle{definition}
\newtheorem{theorem}{Theorem}[section]
\newtheorem{theoremx}{Theorem}

\numberwithin{equation}{section}

\newtheorem*{theorem*}{Theorem}

\newtheorem{corollary}[theorem]{Corollary}
\newtheorem{lemma}[theorem]{Lemma}
\newtheorem{proposition}[theorem]{Proposition}

\newtheorem*{claim*}{Claim}

\theoremstyle{definition}
\newtheorem{definition}[theorem]{Definition}

\newtheorem{conjecture}[theorem]{Conjecture}
\newtheorem{remark}[theorem]{Remark}

\newcommand{\e}{\operatorname{e}}

\newtheoremstyle{TheoremNum}
        {8pt}{8pt}              
        {\upshape}                      
        {}                              
        {\bfseries}                     
        {.}                             
        {.5em}                             
        {\thmname{#1}\thmnote{ \bfseries #3}}
  \theoremstyle{TheoremNum}

\newcommand{\RR}{\mathbb{R}}
\newcommand{\NN}{\mathbb{N}}
\newcommand{\ZZ}{\mathbb{Z}}

\newcommand{\CC}{\mathbb{C}}

\newcommand{\IN}{\operatorname{in}}

\newcommand{\Ass}{\operatorname{Ass}}

\newcommand{\Ht}{\operatorname{ht}}

\newcommand{\HF}{\operatorname{HF}}
\newcommand{\HS}{\operatorname{HS}}

\renewcommand{\b}{\mathfrak{b}}

\newcommand{\ul}[1]{\underline{#1}}

\newcommand{\p}{\mathfrak{p}}
\newcommand{\q}{\mathfrak{q}}

\newcommand{\ps}[1]{\llbracket {#1} \rrbracket}

\renewcommand{\a}{\mathfrak{a}}

\renewcommand{\leq}{\leqslant}
\renewcommand{\geq}{\geqslant}

\newcommand{\LPP}{\operatorname{LPP}}

\DeclareMathOperator{\Tor}{Tor}

\newcommand{\LL}{\mathbb{L}}
\newcommand{\f}{\mathfrak{f}}

\title[]{Quadratic linear strands of prime ideals}
\author{Giulio Caviglia}
\address{Department of Mathematics, Purdue University, 150 N. University Street, West Lafayette, IN 47907-2067, USA}
\email{gcavigli@purdue.edu}
\author{Alessandro De Stefani}
\address{Dipartimento di Matematica, Universit{\`a} di Genova, Via Dodecaneso 35, 16146 Genova, Italy}
\email{alessandro.destefani@unige.it}

\subjclass[2020]{Primary: 13D02, Secondary: 13P10, 13D07}
\keywords{Syzygies, prime ideals, Betti numbers, resolutions}

\begin{document}

\begin{abstract}
We prove sharp estimates on the quadratic strand of the resolution of any homogeneous prime ideal in a standard graded polynomial ring over an arbitrary field. Our bounds only depend on the height of the prime ideal, and they are optimal since for every $h \geq 1$ we show that there exists a prime ideal of height $h$ achieving them. In particular, we show that a prime ideal of height $h$ can contain at most $h^2$ quadratic minimal generators, and that there exists a prime ideal minimally generated by $h^2$ quadrics.
\end{abstract} 

\maketitle

\section{Introduction}

Sharp bounds on graded Betti numbers of algebraic varieties, valid over an arbitrary field and depending on the smallest possible amount of input data, are rarely found across commutative algebra and algebraic geometry. This paper establishes such bounds for the entire quadratic linear strand of the minimal free resolution of any homogeneous prime ideal, with the height as the only input. Further, we prove that these bounds are simultaneously optimal across the whole strand. 

Given their fundamental role in the theory of commutative rings and algebraic varieties, 
prime ideals are often expected to satisfy better properties compared to general ideals in terms of singularities and numerical invariants that measure complexity. Two notable examples of such expectations are the Eisenbud-Goto conjecture relating the Castelnuovo-Mumford regularity, the multiplicity and the height of a prime ideal \cite{EG}, and Green's conjecture on vanishing of Betti numbers of canonical curves in relation to their Clifford index \cite[Conjecture 5.1]{Green}. In \cite{McP}, McCullough and Peeva provide a series of counterexamples to the Eisenbud-Goto conjecture, showing that it can fail quite spectacularly. On the other hand, Green's conjecture has been proved for generic curves by Voisin in characteristic zero \cite{Voisin1,Voisin2} and in certain positive characteristics \cite{AFPRW,RaicuSam}, while it is still wide open in the case of arbitrary curves. 

Another direction of research is to find bounds on certain numerical invariants only relying on partial initial data. This is the case, for instance, of Stillman's question: can one bound the projective dimension (or, equivalently, the Castelnuovo-Mumford regularity) of a homogeneous ideal in a standard graded polynomial ring only in terms of the degrees of its generators? This has been answered affirmatively by Ananyan and Hochster \cite{AH} with the use of very powerful new techniques and a clever inductive method. While the existence of such a bound is remarkable, the nature itself of their proof can only provide bounds which are far from being sharp.

Our main results stand in stark contrast: not only do we establish polynomial bounds depending solely on the height, but these bounds are optimal.

Let $K$ be a field, $S=K[x_1,\ldots,x_N]$ be a polynomial ring endowed with the standard grading, and $P$ be a homogeneous prime ideal containing no linear forms. Let $h$ denote its height, and let $\beta_{i,j}(P) = \dim_K(\Tor_i(P,K)_j)$ denote its $(i,j)$-th graded Betti number. When $K$ is an \emph{algebraically closed field}, a classical result attributed to Guido Castelnuovo establishes that the maximal number of linearly independent quadrics contained in $P$ is $\binom{h+1}{2}$. This bound is remarkable, both because it does not depend on the number of variables of $S$, and because over algebraically closed fields it is sharp, achieved by varieties of minimal degree. An extension of Castelnuovo's estimate to the whole quadratic strand of the resolution was achieved in \cite{HanKwak} (see also \cite{HanKwakLee}); however, in light of Corollary \ref{corollPrime}, these results still crucially require $K$ to be algebraically closed. 

Now let $K$ be \emph{any field}, and $P$ and $h$ be as above. In \cite[Theorem 4.8]{CDS_Prime} we proved that each graded Betti number $\beta_{i,j}(P)$ can be bounded above by a double exponential function depending on the parameters $h,i$ and $j$. Performing a finer analysis on minimal generators of low degree, we also showed that $P$ can contain at most $2h^2+h$ linearly independent quadrics \cite[Theorem 4.6]{CDS_Prime}. We remark that even the existence of a general bound on $\beta_{0,2}(P)$ which does not require $K$ to be algebraically closed and only depends quadratically on $h$ is, by itself, noteworthy. Moreover, we  note that our bounds on graded Betti numbers hold, more generally, for unmixed radical ideals, not necessarily primes. 

In the process of exploring the upper bound for quadrics, in \cite[Example 4.3]{CDS_Prime} we considered the defining ideal of the union of two disjoint linear spaces of codimension $h$; we will denote such an ideal by $\Gamma(h)$, for future reference (see also Subsection \ref{Subsection radical}). Note that $\Gamma(h)$ is an unmixed radical ideal of height $h$ with $h^2$ minimal quadratic generators. In light of the larger bound $\beta_{0,2}(I) \leq 2h^2+h$ produced in \cite{CDS_Prime}, one may expect to find examples with more than $h^2$ minimal quadratic generators. However, the first main goal of this article is to show that not only the minimal generators, but all the graded Betti numbers in the quadratic strand of $\Gamma(h)$ are maximal within the class of unmixed radical ideals of height $h$.

\begin{theoremx}[Theorem \ref{mainthm} \& Remark \ref{remLinear}] \label{mainA} Let $K$ be any field, and $S$ be a standard graded polynomial ring over $K$. If $I \subseteq S$ is a homogeneous unmixed radical ideal of height $h$ (for instance, a prime), then 
\[
\beta_{i,i+2}(I) \leq \binom{2h}{i+2} - 2 \binom{h}{i+2}
\]
for all $i \geq 0$. In particular, $I$ has at most $h^2$ minimal quadratic generators.
\end{theoremx}

The key new insight leading to Theorem \ref{mainA} is the following. Following \cite{CDS_Prime}, one reduces to bound the Betti numbers of an almost complete intersection $\a=\f + (g)$. The first new observation is that, up to a deformation, one may arrange for the complete intersection $\f$ to be monomial. This is the precise hypothesis under which the Lex-Plus-Powers (LPP) conjecture is known to hold, by a theorem of Mermin and Murai \cite{MerminMurai}. However, even after the deformation, the LPP bounds on the Betti numbers of $\a$ are not yet sharp. The final step is to use the fact that, when a radical ideal contains no linear forms, one must have cancellations that bring the LPP estimates down to the sharp values. We emphasize that the LPP conjecture and its predecessor, the Eisenbud-Green-Harris conjecture \cite{EGH,EGH1}, are wide open in general, even for radical ideals and for almost complete intersections. To our knowledge this is the first instance where the LPP conjecture is leveraged, together with a forced cancellation argument, to deduce sharp estimates for a broad class of ideals.

As already pointed out, Theorem \ref{mainA} yields that $\Gamma(h)$ has maximal graded Betti numbers in the quadratic strand of the resolution within the class of unmixed radical ideals of height $h$. Our second main result is to show that the upper bounds of Theorem \ref{mainA} on $\beta_{i,i+2}(-)$ are actually achieved by a prime ideal $\p(h)$ of height $h$ which has the same graded Betti numbers as $\Gamma(h)$, hence proving that the same bounds are sharp also within the class of prime ideals.

\begin{theoremx}[Corollary \ref{corollPrime}] \label{mainB} For any $h\geq 1$ there exists a prime ideal $\p(h)$ of height $h$ such that
\[
\beta_{i,i+2}(\p(h)) = \binom{2h}{i+2} - 2 \binom{h}{i+2}
\]
for all $i \geq 0$. In particular, the upper bounds established in Theorem \ref{mainA} are sharp.
\end{theoremx}

\subsection*{Acknowledgments} The first named author was partially supported by a grant from the Simons Foundation (SFI-MPS-TSM-00013569, G.C.). The second named author was partially supported by the MIUR Excellence Department Project CUP D33C23001110001, PRIN 2022 Project 2022K48YYP, and by INdAM-GNSAGA.

\section{Preliminaries}
Throughout this article, $K$ denotes an arbitrary field, and $S=K[x_1,\ldots,x_N]$ is a polynomial ring with the standard grading, i.e., $\deg(x_i) = 1$ for all $i$. For the purposes of this article, the number $N$ of variables should be thought of as an unknown quantity, possibly very large compared to the rest of the given data. Unless otherwise stated, all ideals and modules involved are homogeneous with respect to this grading. Given a finitely generated $\ZZ$-graded $S$-module $M$, for $j \in \ZZ$ we will write $M_j$ for its degree $j$ component. The Hilbert function $\HF_M:\ZZ \to \NN$ of $M$ is defined as $\HF_M(j) = \dim_K(M_j)$, and its generating series $\HS_M(z) = \sum_{j \in \ZZ} \HF_M(j)z^j \in \ZZ\ps{z}[z^{-1}]$ is called the Hilbert series of $M$. We will let $\beta_{i,j}(M) = \dim_K(\Tor_i^S(M,K))_j$ be the $(i,j)$-th graded Betti number of $M$, and $\beta_i(M) = \sum_{j \in \ZZ}\beta_{i,j}(M)$ be its $i$-th (total) Betti number. We will frequently use that, if $I$ is a homogeneous ideal, then $\beta_{i,j}(I) = \beta_{i+1,j}(S/I)$ for all $i,j \in \ZZ$. Given a $\ZZ$-graded module $M$, and an integer $j$, we let $M[j]$ be the $j$-th shift of $M$, that is, the $\ZZ$-graded $S$-module $\bigoplus_{i \in \ZZ} M[j]_i$ such that $M[j]_i = M_{i+j}$ for all $i \in \ZZ$.

\subsection{Betti numbers of LPP ideals}

Consider the lexicographic order $>_{{\rm lex}}$ on $S$ induced by $x_1 > x_2 >\ldots > x_n$. An ideal $L \subseteq S$ is called a lex ideal if for any two monomials $u,v  \in S$ with $\deg(u) = \deg(v)$ and $u>_{{\rm lex}} v$ we have $u\in  L$ if $v \in L$. Recall that Macaulay's theorem states that, for any homogeneous ideal $I \subseteq S$, there is a lex ideal $L$ with the same Hilbert function as $I$. We start by recalling the Eisenbud-Green-Harris Conjecture, whose aim is to extend Macaulay's theorem by taking into account the data that $I$ contains a regular sequence of prescribed degrees.

\begin{conjecture}{\cite[Eisenbud-Green-Harris]{EGH,EGH1}} \label{conjEGH} Let $I \subseteq S$ be a homogeneous ideal containing a regular sequence of degrees $a_1 \leq \ldots \leq a_h$. There exists a lex ideal $L$ such that $I$ and $(x_1^{a_1},\ldots,x_h^{a_h})+L$ have the same Hilbert function.
\end{conjecture}
One can show that an ideal $(x_1^{a_1},\ldots,x_h^{a_h})+L$ as predicted by the Conjecture, if it exists, is unique. In this case, it is denoted $\LPP^{\ul{a}}(I)$.

The Eisenbud-Green-Harris Conjecture is still open in general, but it is known to be true in several cases. These include when the regular sequence contained in $I$ is monomial \cite{CL}, or when $a_{i+1} \geq \sum_{j=1}^i (a_i-1)$ for all $i=1,\ldots,h-1$ \cite{CM,CDS_EGH}. We refer the interested reader to \cite{EGH_surv} for more details on the subject.

If the Eisenbud-Green-Harris Conjecture holds true, then one can show that $\beta_{0,j}(I) \leq \beta_{0,j}(\LPP^{\ul{a}}(I))$ for all $j \geq 0$ (e.g., see \cite[Proposition 2.14]{EGH_surv}). The Lex-Plus-Powers Conjecture, attributed to Charalambous and Evans, predicts that an analogous inequality holds for all graded Betti numbers, not only for minimal generators.

\begin{conjecture}[Lex-Plus-Powers] \label{conjLPP} Let $I \subseteq S$ be a homogeneous ideal containing a regular sequence of degrees $a_1 \leq \ldots \leq a_h$. If $\LPP^{\ul{a}}(I)$ exists, then $\beta_{i,j}(I) \leq \beta_{i,j}(\LPP^{\ul{a}}(I))$ for all $i,j \in \ZZ$.
\end{conjecture}

Conjecture \ref{conjLPP} is known when $I$ already contains the monomial complete intersection $(x_1^{a_1},\ldots,x_h^{a_h})$ \cite{MerminMurai}, and when ${\rm char}(K)=0$ and $a_{i+1} > \sum_{j=1}^i (a_i-1)$ for all $i=0,\ldots,h-1$ \cite{CavigliaSammartano}. A very concrete advantage of passing to $\LPP^{\ul{a}}(I)$  is that its graded Betti numbers can be computed explicitly using some combinatorial data. In this article, we will use a formula due to Murai \cite{Murai} in order to compute some of them for an explicit ideal (see Proposition \ref{propLPP}). In order to do so, we will first recall some notation from \cite{Murai}, which we will then  adapt to our setup.

Throughout, we will let $\binom{a}{b} := 0$ whenever $a<b$ or $a<0$ or $b<0$. Also, we let $\binom{a}{0} := 1$ for all $a \geq 0$. Let $a_1,\ldots,a_N \in \ZZ_{\geq 2} \cup \{\infty\}$. Using the standard arithmetic conventions, we will further assume that $a_1 \leq a_2 \leq \ldots \leq a_N$. For a variable $x_i$ we will let $x_i^{\infty} := 0$, and we will consider the ideal $J(\ul{a})=(x_1^{a_1},\ldots,x_N^{a_N})$ in $S$. 

\begin{definition} Let $\ul{a}=(a_1,\ldots,a_N)$ with $a_i \in \ZZ_{\geq 2} \cup \{\infty\}$. An ideal $I \subseteq S$ is said to be $\ul{a}$-LPP if there exists a lex ideal $L \subseteq S$ such that $I=J(\ul{a})+L$.
\end{definition}

Given a monomial $v=x_{i_1}^{b_1} \cdots x_{i_t}^{b_t} \notin J(\ul{a})$ with $b_1,\ldots,b_t>0$ and $i_1>i_2>\ldots>i_t$, we let $k(v) = \min\{i \mid b_i<a_i-1\}$ if $v \ne x_{i_1}^{a-1} \cdots x_{i_t}^{a-1}$, and $k(v)=t+1$ otherwise. Moreover, for $i \geq 0$, we let $A_i(v) = \sum_{\ell=1}^{k(v)} \binom{i_{\ell}-1}{i-\ell}$, with $i_{t+1} :=0$. For an $\ul{a}$-LPP ideal $I$, we let $M^{\ul{a}}(I)$ be the set of monomials of $S$ that do not belong to $J(\ul{a})$. By \cite[Proposition 2.1]{Murai} we have
\begin{equation}
\label{Murai} \beta_{i,i+j}(S/I) = \sum_{\substack{u \in M^{\ul{a}}(I) \\ \deg(u)=j+1}} A_i(u) - \sum_{\substack{u \in M^{\ul{a}}(I) \\ \deg(u)=j}} \left(\binom{N}{i} - A_{i+1}(u)\right) + \beta_{i,i+j}(S/J(\ul{a})).
\end{equation}

\subsection{An explicit Betti numbers calculation} \label{SubsectionLPP}

Let $a>2$ and $h >1$ be integers, and set $J=(x_1^a,\ldots,x_h^a)$. Note that $J = J(\ul{a})$ for $\ul{a} = (a_1,\ldots,a_N)$, with $a_i=a$ for $1 \leq i \leq h$ and $a_i = \infty$ for $i>h$. Let $u=(x_1 \cdots x_h)^{a-1}$ and, for $j=1,\ldots,h$, let $v_j = (x_1 \cdots x_{h-1})^{a-1}x_h^{a-2}x_{h+1}x_{h+j}$. Set 
\[
\LL(a,h):=J+(u)+(v_j \mid j=1,\ldots,h).
\]

\begin{proposition} \label{propLPP} Let $a>2$ and $h >1$ be integers. With the notation introduced above, if we let $D=(a-1)h$ then we have
\[
\beta_{i,i+j}(S/\LL(a,h)) = \begin{cases} \beta_{i,i+j}(S/J) & \text{ if } j < D-1 \\
0 & \text{ if } i>h \text{ and } j=D-1 \\
\binom{h}{i-1} & \text{ if } i \leq h \text{ and } j=D-1 \\
\binom{2h}{i}-\binom{h}{i} & \text{ if } j=D.
\end{cases}
\]
\end{proposition}
\begin{proof}
Let $\ul{a} = (a_1,\ldots,a_N)$, with $a_i=a$ for $1 \leq i \leq h$ and $a_i = \infty$ for $i>h$. It is straightforward to verify that $\LL:=\LL(a,h)$ is a $\ul{a}$-LPP ideal. In particular, we can apply (\ref{Murai}) to compute the Betti numbers of $\LL$. For $j<D-1$, one directly gets $\beta_{i,i+j}(S/\LL) = \beta_{i,i+j}(S/J)$ for all $i \geq 0$ since there are no monomials in $\LL \smallsetminus J$ of degree less than $D$. In terms of (\ref{Murai}), this can be stated as $\{v \in M^{\ul{a}}(\LL) \mid \deg(v)<D\} = \emptyset$. Next, applying (\ref{Murai}) with $j=D-1$ we get
\[
\beta_{i,i+D-1}(S/\LL) = A_i(u) = \sum_{\ell=1}^{h+1} \binom{h-\ell}{i-\ell} = \begin{cases} \binom{h}{i-1} & \text{ if } i \leq h \\ \\  0 & \text{ otherwise}
\end{cases}
\]
since $u$ is the only monomial of degree $D$ in $M^{\ul{a}}(\LL)$. In degree $D+1$, $M^{\ul{a}}(\LL)$ consists of $\{v_j \mid j=1,\ldots,h\}\cup\{ux_{h+j} \mid j=1,\ldots,N-h\}$.  The formula (\ref{Murai}) then gives
\begin{align*}
\beta_{i,i+D}(S/\LL) & = \sum_{j=1}^h A_i(v_j) + \sum_{j=1}^{N-h} A_i(ux_{h+j}) - \binom{N}{i} + A_{i+1}(u) + \beta_{i,i+D}(S/J) \\
& = \sum_{j=1}^h \binom{h+j-1}{i-1} + \sum_{j=1}^{N-h} \binom{h+j-1}{i-1} - \binom{N}{i} + \sum_{\ell=1}^h \binom{h-\ell}{i+1-\ell} + \beta_{i,i+D}(S/J) \\ 
& = \left(\binom{2h}{i} - \binom{h}{i}\right) + \left(\binom{N}{i} - \binom{h}{i}\right) - \binom{N}{i} +\sum_{\ell=1}^h \binom{h-\ell}{i+1-\ell} + \beta_{i,i+D}(S/J) \\ 
& = \binom{2h}{i} - 2 \binom{h}{i} +\sum_{\ell=1}^h \binom{h-\ell}{i+1-\ell} + \beta_{i,i+D}(S/J).
\end{align*}
If $i \ne h$, then $\beta_{i,i+D}(S/J) = 0$, and $\sum_{\ell=1}^h \binom{h-\ell}{i+1-\ell}=  \binom{h}{i}$. On the other hand, if $i=h$ then $\beta_{h,h+D}(S/J) = 1$ and $\sum_{\ell=1}^h \binom{h-\ell}{i+1-\ell}=  0$. In conclusion, for all $i \geq 0$ we have
\[
\beta_{i,i+D}(S/\LL) = \binom{2h}{i} - \binom{h}{i}. \qedhere
\]
\end{proof}

\section{Upper bounds on quadratic linear strands}  \label{Section bound}

We recall that an ideal $I \subsetneq S = K[x_1,\ldots,x_N]$ is said to be unmixed if $\dim(S/I) = \dim(S/\p)$ for all $\p \in \Ass(S/I)$. In particular, an unmixed ideal has no embedded associated primes, and we must have $\Ht(\p) = N-\dim(S/I)$ for all $\p \in \Ass(S/I)$. We will now reduce the problem of studying graded Betti numbers of an unmixed homogeneous radical ideal $I$ of height $h$ to that of studying graded Betti numbers of a homogeneous almost complete intersection $\a$, that is, an ideal of height $h$ generated by $h+1$ homogeneous polynomials. While this reduction is quite standard (e.g., see \cite[Lemma 4.1]{CDS_Prime}), for our purposes one needs to have good control on the degree of the last generator of $\a$. To this end, we recall the following result.
\begin{theorem} \label{thmChardin}
Let $K$ be an infinite field, and $S=K[x_1,\ldots,x_N]$. Let $I \subseteq S$ be an unmixed radical ideal of height $h<N$, and $f_1,\ldots,f_h$ be a homogeneous regular sequence of degrees $d_1 \leq \ldots \leq d_h$ contained in $I$. If we let $J=(f_1,\ldots,f_h)$, there exists a homogeneous element $g$ of degree $D = \sum_{i=1}^h(d_i-1)$ such that $I=J:g$.
\end{theorem}
\begin{proof}
By a result of Chardin, \cite[Theorem 27 \& Corollary 28]{Chardin}, for each minimal prime $\p_i$ of $I$ there exists $g_i$ of degree at most $D$ such that $\p_i = J:g_i$. We note that, if $\q_i$ denotes the $\p_i$-primary component of $J$, then each $g_i$ does not belong to $\q_i$, but it belongs to every other primary component of $J$. If we let $g=g_1 + \ldots + g_h$, then $g$ has degree at most $D$, and $\q_i: g = \q_i:g_i = \p_i$. It follows that $J:g = \bigcap_i (\q_i:g) = \bigcap_i \p_i = I$. Since $K$ is infinite, and $h < N$, we can find a sufficiently general linear form $\ell$ such that $J:\ell = J$. If follows that $J:\ell^{D-\deg(g)}g = I$, and $\ell^{D-\deg(g)}g$ has degree $D$ as desired.
\end{proof}

\begin{lemma} \label{lemmaACI} Let $f_1,\ldots,f_h$ be a homogeneous regular sequence of degrees $2 \leq d_1 \leq \ldots \leq d_h$, and $\f=(f_1,\ldots,f_h)$. Let $g \in S$ be a homogeneous element of degree $D = \sum_{i=1}^h (d_i-1)$ such that $I=\f:g$ is a proper ideal containing no linear forms. Let $\a = \f+(g)$. We have 
\[
\beta_{i,j}(S/I) = \begin{cases} \beta_{i+1,j+D}(S/\a) & \text{ if } (i,j) \ne (h-1,h) \\
\beta_{h,h+D}(S/\a) -1 & \text{ if } (i,j) = (h-1,h).
\end{cases}
\]
\end{lemma}
\begin{proof}
Let $n \geq 0$ be an integer. Applying the functor $-\otimes_S K$ to the graded short exact sequence
\[
\xymatrix{
0 \ar[rr] && S/I[-D] \ar[rr] && S/\f \ar[rr] && S/\a \ar[rr] && 0
}
\]
and specializing in degree $n+D$ gives a graded long exact sequence
\[
\xymatrixcolsep{4mm}
\xymatrix{
\ldots \ar[r] &\Tor_{i+1}^S(S/\f,K)_{n+D} \ar[r] & \Tor_{i+1}^S(S/\a,K)_{n+D} \ar[r] & \Tor_{i}^S(S/I,K)_{n} \ar[r] & \Tor_i^S(S/\f,K)_{n+D} \ar[r] &  \ldots.
}
\]
Since $\f$ is a complete intersection generated by forms of degrees $d_1,\ldots,d_h$ we have that $\Tor_i(S/\f,K)_{n+D} = 0$ for all $(i,n) \ne (h,h)$. Thus, for $(i,n) \notin \{(h-1,h),(h,h)\}$ we obtain isomorphisms $\Tor_i(S/I,K)_n \cong \Tor_{i+1}(S/\a,K)_{n+D}$, so that $\beta_{i,n}(S/I) = \beta_{i+1,n+D}(S/\a)$. Since $I$ does not contain any linear forms, we must have $\Tor_h(S/I,K)_{h} = 0$, and this forces $\Tor_{h+1}^S(S/\a,K)_{h+D}=0$ as well, so equality on Betti numbers holds also in this case. For the remaining choice $(i,n) = (h,h)$, the long exact sequence of $\Tor$ modules gives
\[
\xymatrix{
0 \ar[r] & \Tor_h^S(S/\f,K)_{h+D} \ar[r] & \Tor_{h}^S(S/\a,K)_{h+D} \ar[r] & \Tor_{h-1}^S(S/I,K)_{h} \ar[r] & 0.
}
\]
Since $\Tor_h^S(S/\f,K)_{h+D} \cong K$, we conclude that $\beta_{h-1,h}(S/I) = \beta_{h,h+D}(S/\a)-1$.
\end{proof}

We are now ready to prove our first main theorem.

\begin{theorem} \label{mainthm}
Let $K$ be any field, and $S=K[x_1,\ldots,x_N]$ with the standard grading. Let $I \subseteq S$ be a homogeneous unmixed radical ideal of height $h$, containing no linear forms. For all $i \geq 0$ we have
\[
\beta_{i,i+2}(I) \leq \binom{2h}{i+2} - 2 \binom{h}{i+2}.
\]
In particular, $I$ has at most $h^2$ quadratic minimal generators.
\end{theorem}
\begin{proof} 
The case $h=1$ is trivial, so we will henceforth assume $h>1$. Since replacing $K$ with a purely transcendental extension preserves our assumptions and does not change Betti numbers, we may assume without loss of generality that $K$ is infinite. Consider the degree-revlex order on $S$, and let $\IN(-)$ denote the initial ideal with respect to this order. After possibly performing a sufficiently general change of coordinates, we may assume that $x_N,\ldots,x_{N-h+1}$ forms a system of parameters for $S/I$. By standard properties of revlex-type orders (e.g., see \cite[15.7]{Eisenbud}), we have that $\IN(I+(x_N,\ldots,x_{N-h+1})) = \IN(I) + (x_N,\ldots,x_{N-h+1})$. Since $S/(\IN(I) + (x_N,\ldots,x_{N-h+1}))$ is Artinian, it follows that $\IN(I)+ (x_N,\ldots,x_{N-h+1}))$ contains $(x_1^a,\ldots,x_h^a)$ for some $a > 0$. Since we considered the degree-revlex order, this implies that $\IN(I)$ itself contains $(x_1^a,\ldots,x_h^a)$. Note that one may take $a$ arbitrarily large; for instance, we will need $a \gg h$, in a way that will be made more precise later in this proof. Since $(x_1^a,\ldots,x_h^a) \subseteq \IN(I)$, we have that $I$ contains a regular sequence $f_1,\ldots,f_h$ of forms of degree $a$ such that $\IN(f_i) = x_i^a$. Let $\f=(f_1,\ldots,f_h)$. By Theorem \ref{thmChardin} there exists $g \in S$ of degree $D=h(a-1)$ such that $I=\f:g$. Now let $\a=\f+(g)$; thanks to Lemma \ref{lemmaACI}, it suffices to estimate its Betti numbers. Observe that $\IN(\a) \supseteq (\IN(f_1),\ldots,\IN(f_h),\IN(g)) = J+(\IN(g))$. Since $x_1^a,\ldots,x_h^a$ forms a regular sequence, and by degree reasons, we must have that $\IN(\a)$ and $\f+(\IN(g))$ coincide in degrees less than or equal to $D$. In the calculation of a Gr{\"o}bner basis of $I$ in degree $D+1$, Buchberger's algorithm yields at most $h$ new polynomials, coming from the S-pairs between $f_i$ and $g$ for $i=1,\ldots,h$. Thus, the generators of $\IN(\a)$ in degree $D+1$ are monomials $u_1,\ldots,u_t$ with $t \leq h$. Consider the monomial ideal $\b=(x_1^a,\ldots,x_h^a,\IN(g),u_1,\ldots,u_t)$ and note that, since $\IN(\a)$ and $\b$ coincide in all degrees less than or equal to $D+1$, we must have $\beta_{i,i+j}(\IN(\a)) = \beta_{i,i+j}(\b)$ for all $(i,j)$ with $j \leq D+1$. By semi-continuity of Betti numbers when passing to initial ideals we conclude that $\beta_{i,i+j}(\a) \leq \beta_{i,i+j}(\b)$ for all $(i,j)$ with $j \leq D+1$. Moreover, if we let $\ul{a}=(a_1,\ldots,a_N)$ with $a_i=a$ for $i\leq h$ and $a_i=\infty$ for $i>h$, by \cite{MerminMurai} we have that $\beta_{i,j}(\b) \leq \beta_{i,j}(\LPP^{\ul{a}}(\b))$. Let $J=(x_1^a,\ldots,x_h^a)$, $u=(x_1 \cdots x_h)^{a-1}$ and $v_j=(x_1 \cdots x_{h-1})^{a-1}x_h^{a-2}x_{h+1}x_{h+j}$ for $j \geq 1$. Since $\HF(\LPP^{\ul{a}}(\b);j) = \HF(J;j)$ for all $j<D$ and $\HF(\LPP^{\ul{a}}(\b);D) = \HF(J;D)+1$, we have that $\LPP^{\ul{a}}(\b)$ coincides with $J+(u)$ in degrees at most $D$. Now note that $\beta_{i,D+1}(S/\LPP^{\ul{a}}(\b)) = 0$ for $i \geq 3$. On the other hand, by \cite{Murai} we see that $\beta_{2,D+1}(S/\LPP^{\ul{a}}(\b)) = h$. By \cite{MerminMurai} and standard properties regarding cancellation of Betti numbers of initial ideals, we must have
\[
\beta_{1,D+1}(S/\LPP^{\ul{a}}(\b)) = \beta_{2,D+1}(S/\LPP^{\ul{a}}(\b)) = h,
\]
so that $\LPP^{\ul{a}}(\b)$ must have $h$ minimal generators in degree $D+1$. Thus, if we let
\[
\LL = J+(u)+(v_j \mid j=1,\ldots,h),
\]
then $\LPP^{\ul{a}}(\b)$ and $\LL$ coincide in degrees less than or equal to $D+1$. In particular, we have that $\beta_{i,i+j}(\LPP^{\ul{a}}(\b)) = \beta_{i,i+j}(\LL)$ for all $(i,j)$ with $j \leq D+1$.  Note that, with the notation introduced in Subsection \ref{SubsectionLPP}, we have $\LL=\LL(a,h)$. By Proposition \ref{propLPP} we then have that
\[
\beta_{i,i+D+1}(\LPP^{\ul{a}}(\b)) = \binom{2h}{i+1} - \binom{h}{i+1}
\]
and
\[
\beta_{i,i+D}(\LPP^{\ul{a}}(\b)) = \begin{cases} \binom{h}{i} & \text{ if } i \leq h-1 \\ \\
0 & \text{ otherwise.}
\end{cases}
\]
Moreover, choosing $a \gg h$ guarantees that $\beta_{i,i+j}(\LPP^{\ul{a}}(\b)) = \beta_{i,i+j}(\LL) = \beta_{i,i+j}(J) = 0$ for all $(i,j)$ with $i \geq 0$ and $j=D-1$, and for all $(i,j)$ with $i \geq h$ and $j \leq D$. Since $\beta_{i,i+D}(\a)=\beta_{i-1,i}(I)=0$ for all $i \geq 1$, using \cite{MerminMurai} and cancellation properties for initial ideals, we must have a consecutive cancellation between $\beta_{i+1,i+D}(\LL)$ and $\beta_{i,i+D+1}(\LL)$ for all $i \geq 0$. In particular, we see that for all $i \geq 0$ we must have
\[
\beta_{i,i+1+D}(\a) \leq \beta_{i,i+1+D}(\LL) - \beta_{i+1,i+D}(\LL) = \begin{cases}
\binom{2h}{i+1} - 2\binom{h}{i+1} & \text{ if } i \ne h-1  \\ \\
\binom{2h}{h} - 1 & \text{ if } i=h-1.
\end{cases} 
\]
Using Lemma \ref{lemmaACI}, for $i \geq 0$ we get
\[
\beta_{i,i+2}(I) = \begin{cases} \beta_{i+1,i+2+D}(\a) \leq \binom{2h}{i+2} - 2 \binom{h}{i+2} & \text{ if } i \ne h-2 \\ \\
\beta_{h-1,h+D}(\a) - 1 \leq \binom{2h}{h} -2 = \binom{2h}{h} - 2\binom{h}{h} & \text{ if }i=h-2.
\end{cases}
\]
In any case, we conclude that $\beta_{i,i+2}(I) \leq \binom{2h}{i+2}-2\binom{h}{i+2}$, as desired.
 \end{proof}
 
\begin{remark} \label{remLinear} We would like to remark that, in Theorem \ref{mainthm}, the assumption that $I$ contains no linear forms can be removed. In fact, suppose that $t=\HF(I;1)>0$. Then, after possibly performing a change of coordinates we may assume that $I=I'+(x_{N-t+1},x_{N-t+2},\ldots,x_N)$ for some ideal $I'$ which is still homogeneous, unmixed, radical, of height $h-t$ and containing no linear forms. Let $I''=(x_{N-t+1},x_{N-t+2},\ldots,x_N)$. Since $S/I = S/I' \otimes_S S/I''$ and $\Tor_1^S(S/I',S/I'')=0$, by rigidity of $\Tor$ over $S$ and because $S/I''$ is resolved by a Koszul complex we obtain that 
\[
\beta_{i,i+1}(S/I) = \sum_{j=0}^i \beta_{j,j+1}(S/I') \beta_{i-j,i-j}(S/I'') = \sum_{j=0}^i \beta_{j,j+1}(S/I') \binom{t}{i-j}.
\]
By Theorem \ref{mainthm} we have that $\beta_{j,j+1}(S/I') \leq \binom{2(h-t)}{j+1} - 2\binom{h-t}{j+1}$ for all $j \geq 1$. 
A direct calculation now shows that
\begin{align*} 
\beta_{i,i+1}(S/I) & \leq \sum_{j=1}^i \left[\binom{2(h-t)}{j+1} - 2\binom{h-t}{j+1}\right]\binom{t}{i-j} \\
& = \sum_{j=2}^{i+1} \binom{2(h-t)}{j} \binom{t}{i+1-j} - 2 \sum_{j=2}^{i+1} \binom{h-t}{j} \binom{t}{i+1-j} \\
& = \left[\binom{2h-t}{i+1} - \binom{t}{i+1} - 2(h-t)\binom{t}{i}\right] - 2 \left[\binom{h-t}{i+1} - \binom{t}{i+1} - (h-t)\binom{t}{i}\right] \\
& = \binom{2h-t}{i+1}  + \binom{t}{i+1} - 2 \binom{h}{i+1}\\
& \leq \binom{2h}{i+1}-2\binom{h}{i+1}.
\end{align*}
 \end{remark}
 
\begin{remark}
We conclude this section by noting that, although a similar strategy could in principle be employed to obtain bounds for higher strands of the resolution, it yields estimates that are significantly weaker than those established in Theorem \ref{mainthm} for the quadratic strand. Indeed, the proof of Theorem \ref{mainthm} relies crucially on controlling cancellations within the resolution of the initial ideal and its LPP ideal. While this process is manageable for the quadratic strand, it rapidly becomes intractable for higher strands.
\end{remark}

\section{A prime ideal achieving the upper bound}
The goal of this section is to show that the bounds obtained in Theorem \ref{mainthm} are sharp and, in fact, are achieved by an ideal which is not only radical, but even prime.

\subsection{A radical example} \label{Subsection radical} Let $K$ be a field, and $S=K[u_1,\ldots,u_h,v_1,\ldots,v_h]$. Let $\Gamma(h):=(u_1,\ldots,u_h) (v_1,\ldots,v_h) = (u_1,\ldots,u_h) \cap (v_1,\ldots,v_h)$. From the short exact sequence
\[
\xymatrix{
0 \ar[r] &  S/\Gamma(h) \ar[r] & \displaystyle \frac{S}{(u_1,\ldots,u_h)} \oplus \frac{S}{(v_1,\ldots,v_h)} \ar[r] & \displaystyle \frac{S}{(u_1,\ldots,u_h,v_1,\ldots,v_h)} \cong K \ar[r] & 0 
}
\]
we deduce that $\HS_{S/\Gamma(h)}(z) = \frac{2}{(1-z)^h} - 1$, so that $\HS_{\Gamma(h)}(z) = \frac{(1-z)^{2h}-2(1-z)^h+1}{(1-z)^{2h}}$.

For $i \geq 0$ let $\beta_i=\beta_i(\Gamma(h))$. Since $\Gamma(h)$ is generated by quadrics, and it has linear resolution (e.g., see \cite[Theorem 3.1]{ConcaHerzog}), we have 
\[
\HS_{\Gamma(h)}(z) = \sum_{i \geq 0}(-1)^i \frac{\beta_i z^{i+2}}{(1-z)^{2h}} = \left(\sum_{i \geq 0} (-1)^i \beta_iz^{i+2}\right) \HS_S(z).
\]
Replacing $z$ with $-z$, and using the previous information, we obtain 
\[
\sum_{i \geq 0} \beta_iz^{i+2} = \frac{\HS_{\Gamma(h)}(-z)}{\HS_S(-z)} = (1+z)^{2h}-2(1+z)^h+1.
\]
In particular, the $i$-th Betti number of $\Gamma(h)$, i.e., the coefficient of $z^{i+2}$, is $\binom{2h}{i+2} - 2 \binom{h}{i+2}$. 

\subsection{A prime example}
Let $R=\RR[x_j,ix_j \mid j=1,\ldots,h]$, with $i^2=-1$ the imaginary unit. Note that $R$ is a domain, and that $\dim(R)=h$ since it is integral over $\RR[x_1,\ldots,x_h]$. If we let $S=\RR[y_j,z_j \mid j=1,\ldots,h]$, then the map $y_j \mapsto x_j$ and $z_j \mapsto ix_j$ gives a presentation $R \cong S/\p(h)$ for some prime ideal $\p(h)$ of $S$ of height $h$. Consider the ideal 
\[
I=(y_jz_k-y_kz_j \mid 1 \leq j < k \leq h) + (y_jy_k+z_jz_k \mid 1 \leq j \leq k \leq h).
\]
A direct calculation shows that $I \subseteq \p(h)$. Since $y_1^2+z_1^2,\ldots,y_h^2+z_h^2$ forms a regular sequence inside $I$, we must have $\Ht(I)=h$ and $\p(h)$ must then be a minimal prime over $I$. Now let $T=\CC[y_j,z_j \mid j=1,\ldots,h] \cong S \otimes_\RR \CC$ be the base change to $\CC$. Performing the change of coordinates $y_j+iz_j \mapsto u_j$ and $y_j-iz_j \mapsto v_j$, we see that 
\begin{align*}
IT & \cong(u_jv_k-u_kv_j \mid 1 \leq j<k \leq h) + (u_jv_k + u_kv_j \mid 1 \leq j \leq k \leq h) \\
 & = (u_jv_k \mid 1 \leq j,k \leq h) = (u_1,\ldots,u_h)  (v_1,\ldots,v_h).
\end{align*}
Since the multiplicity does not change after extending scalars, we see that $\e(S/I) = \e(T/IT) = 2$. Moreover, since $T/IT$ is unmixed, so must be $S/I$. If we had $I \subsetneq \p(h)$, then $\e(S/\p(h))=1$ would be forced, and $S/\p(h) \cong R$ would then be regular. But this is a contradiction since $R$ is standard graded of dimension $h$, and there are no $\RR$-linear relations between the $2h$ generators of $R$ as an $\RR$-algebra. Hence $I=\p(h)$. 

Since the  graded Betti numbers do not change if we extend scalars, the results of this section yield the following conclusion.
\begin{corollary} \label{corollPrime} For any $h \geq 1$ there exists a homogeneous prime ideal $\p(h)$ of height $h$ with
\[
\beta_{i,i+2}(\p(h)) = \binom{2h}{i+2} - 2 \binom{h}{i+2}
\]
for all $i \geq 0$. In particular, the upper bounds established in Theorem \ref{mainA} are sharp within the class of prime ideals.
\end{corollary}

\bibliographystyle{abbrv}
\bibliography{References}

\end{document}